\documentclass[11pt]{amsart}
\usepackage{graphicx, xcolor}
\usepackage{amssymb}
\numberwithin{equation}{section}

\def\N{\mathbb{N}}
\def\R{\mathbb{R}}

\newcommand{\rdbrack}{]\!]}
\newcommand{\ldbrack}{[\![}

\def\epsilon{\varepsilon}
\def\e{\varepsilon}

\newcommand\br{\begin{rem}}
\newcommand\er{\end{rem}}
\newcommand\bp{\begin{pmatrix}}
\newcommand\ep{\end{pmatrix}}
\newcommand\be{\begin{equation}}
\newcommand\ee{\end{equation}}
\newcommand\ba{\begin{equation}\begin{aligned}}
\newcommand\ea{\end{aligned}\end{equation}}

\newtheorem{theorem}{Theorem}[section]

\newtheorem{corollary}[theorem]{Corollary}

\newtheorem{remark}[theorem]{Remark}
\newtheorem{ans}[theorem]{Definition}

\title{On the metastable behavior of solutions to a class of parabolic systems} 

\begin{document}

\maketitle

\begin{center}
MARTA STRANI\footnote{Universit\`a di Milano Bicocca, Dipartimento di Matematica e Applicazioni, Milano (Italy). E-mail address: \texttt{marta.strani@unimib.it}, \texttt{martastrani@gmail.com}.}
\end{center}

\vskip1cm

\begin{abstract}
In this paper we describe the metastable behavior of solutions to a class of parabolic systems. In particular, we improve some results contained in \cite{MS}  by using different techniques to describe the slow motion of the internal layers. Numerical simulations illustrate the results.
\end{abstract}

\begin{quote}\footnotesize\baselineskip 14pt 
{\bf Key words:} 
Metastability, slow motion, internal layers, reaction-diffusion systems. 
 \vskip.15cm
\end{quote}

\begin{quote}\footnotesize\baselineskip 14pt 
{\bf AMS subject classifications:} 35B25, 35B36, 35B40,  35K45
 \vskip.15cm
\end{quote}

\pagestyle{myheadings}
\thispagestyle{plain}
\markboth{M.STRANI}{Metastability for parabolic systems}

\section{Introduction}

The slow motion of internal interfaces has been widely studied for a large class of evolutive PDEs.
Such phenomenon is known as {\it metastability}. The qualitative features of a metastable dynamics are the following: through a transient process, a pattern of internal layers is formed from initial data over a $\mathcal O(1)$ time interval. However, once this pattern is formed, the subsequent motion of the internal layers is exponentially slow, converging to their asymptotic limit. As a consequence, two different time scales emerge: for short times, the solution are close to some non-stationary state; subsequently, they drifts towards the equilibrium solution with a speed rate that is exponentially small.
\vskip0.2cm
Many fundamental partial differential equations, concerning different areas, exhibit such behavior. Slow motion of internal layers appears in the study of viscous scalar conservation laws, studied, for example, in \cite{LafoOMal94, LafoOMal95, MS, ReynWard95}. These equations emerge in different fields of the fluid mechanics, and they are used as simplified models to study, for example, the gas dynamics and the traffic flow. The metastable dynamics of interfaces arises also in the study of the so called generalized Burgers equation, describing the dynamics of an upwardly propagating flame-front in a vertical channel  (see \cite{BerKamSiv01,  Str13, SunWard99}). Another area where the phenomenon of slow motion on internal interfaces appears is the one of phase transition problems, described by the Allen-Cahn and Cahn-Hilliard equation, with the fundamental contributions \cite{AlikBatFus91, CarrPego89, FuscHale89,Pego89}. Two more recent contributions are the references \cite{OttRez06,Str14}. Depending on the assumptions, these equations describe different models in mathematics and physics; they find applications in neurophysics, biophysics, population genetics, and, especially, in the process of phase separation of fluids and/or metals.
Finally, we quote here the area of relaxation schemes \cite{JinXin95,Str12}, extensively used to approximate systems by a nearby problem  that resemble the original system with a small dissipative correction, and that is easier to solve.

\vskip0.2cm
The study of slow motion for parabolic equations dates back the work of Kreiss and Kreiss \cite{KreiKrei86}; here the authors concern with a viscous scalar conservation law in a bounded domain of the real line, proving that the eigenvalues of the linearized operator around the (unique) steady state of the equation have the following distribution
\begin{equation*}
\lambda_1^\varepsilon= \mathcal O(e^{-1/\varepsilon}), \qquad \lambda_k^\varepsilon \leq -\frac{c}{\varepsilon}, \quad \forall \, k \geq 2.
\end{equation*}
This precise decomposition of the spectrum translates into a slow motion for the solutions of the equation, since, for large times,  their dynamics is described by terms of order $e^{\lambda_1^\varepsilon t}$.

\vskip0.2cm
Starting from \cite{KreiKrei86}, there are a number of papers that deal with the slow dynamics of solution to different partial differential equations. The techniques used are very different, depending on the equations under consideration, but usually the common aim is to find an equation for the position of the layers describing their slow dynamics. 

One recent contribution is the reference \cite{MS}, where the authors develop a general technique to describe the slow motion of internal layers for a general class of parabolic systems. In particular, by linearizing the original system around an ``approximate manifold" of stationary solutions, they end up with a couple system for the position of the layer, named $\xi$, and the perturbation $v$, defined as the difference between the solution and an element of the manifold.  They subsequent describe the slow dynamics of the solutions for the approximating system (see, \cite[Theorem 2.1]{MS}) for the couple $(\xi, v)$, where the higher order terms  are canceled out; indeed, handling with the complete system brings into the analysis also the specific form of the nonlinear terms in $v$. For example, in the case of viscous conservation laws, this means the appearance of a first order space derivative, and a rigorous result needs an additional bound (namely, a bound for the $H^1$ norm of the solution instead that for the $L^2$ norm). The key point of their approach is a spectral analysis of the linearized operator $\mathcal L^\varepsilon_\xi$, arising from the linearization of the original system.
\vskip0.2cm
In this paper we obtain some new results describing the metastable behavior of a class of parabolic systems. 

Firstly, we obtain a new result concerning the analysis of the so called ``quasi-linearized" system for the couple $(\xi,v)$, already studied in \cite[Theorem 2.1]{MS}. In particular, by using a different approach based on the  {\it theory of stable families of generators}, we prove a slightly different estimate for the perturbation $v$ that still states that $v$ has a very fast decay in time, up to a reminder that is small in $\varepsilon$, as already pointed out in \cite{MS}.
Thanks to this new strategy, we are able to remove one of the hypotheses stated in \cite[Theorem 2.1]{MS} and concerning the convergence of the series of the eigenfunctions of the linearized operator $\mathcal L^\varepsilon_\xi$. 
 
In the last part of the paper, we state and prove a new result describing the behavior of the solutions to the complete system for $(\xi, v)$, where also the nonlinear terms in $v$ are taken into account. In particular, in the case of systems of conservation laws, we obtain a bound for the $H^1-$norm of the perturbation $v$ that will be used to decoupled the system, in order to obtain a precise estimate for the size of the parameter $\xi$, describing the slow motion of solutions. Obviously, this makes the theory much more complete with respect to \cite{MS}, since we are able to provide an estimate for the solution to the complete system, that suites the behavior of the solutions to the original system.

 Precisely, we will show that the perturbation $v$ can be decomposed as $v=z+R$, where
 \begin{equation*}
 |z|_{{}_{H^1}} \, \leq |v_0|_{{}_{H^1}} e^{- c t}, \quad {\rm and} \quad 
 |R|_{{}_{H^1}}\,\leq C\,
		\left\{ \varepsilon^\delta \, |v_0|^2_{{}_{H^1}}+\varepsilon^{1-\delta}\right\}.
 \end{equation*}
 for $\delta \in (0,1)$. Hence, as in the ``quasi linear" case, $v$ has a very fast decay in time up to a reminder that is small in $\varepsilon$, since it behaves like $\varepsilon^\delta$.

\section{General Framework}

Given a bounded interval $I=[a,b] \subset \R$, we consider the Cauchy  problem
\begin{equation}\label{GS}
\partial_t u = \mathcal F^\varepsilon[u] , \qquad u(x,0)=u_0(x),
\end{equation} 
 where the unknown $u: [0, +\infty) \to [L^2(I)]^n$, and $\mathcal F^\varepsilon[u]$ is an ordinary one dimensional differential operator that depends singularly on the parameter $\varepsilon$, meaning that $\mathcal F^0$ is of lower order.
 System \eqref{GS} is to be considered for $x \in [a,b]$, and it is complemented with appropriate boundary conditions.
 
 To fix the ideas, we can think about systems of conservation laws, or reaction-diffusion type systems. In these cases, the solutions to \eqref{GS} exhibit a metastable behavior; precisely, starting from an initial datum located far from the stable steady steady to \eqref{GS}, i.e. a solution to $\mathcal F^\varepsilon[u]=0$, we expect the solution to develop into a layered function in a time interval of order $1$, before converging to its asymptotic limit in a exponentially long time interval (see Fig. 1 and Fig. 2).

 \begin{figure}
\centering
\includegraphics[width=14cm,height=10cm]{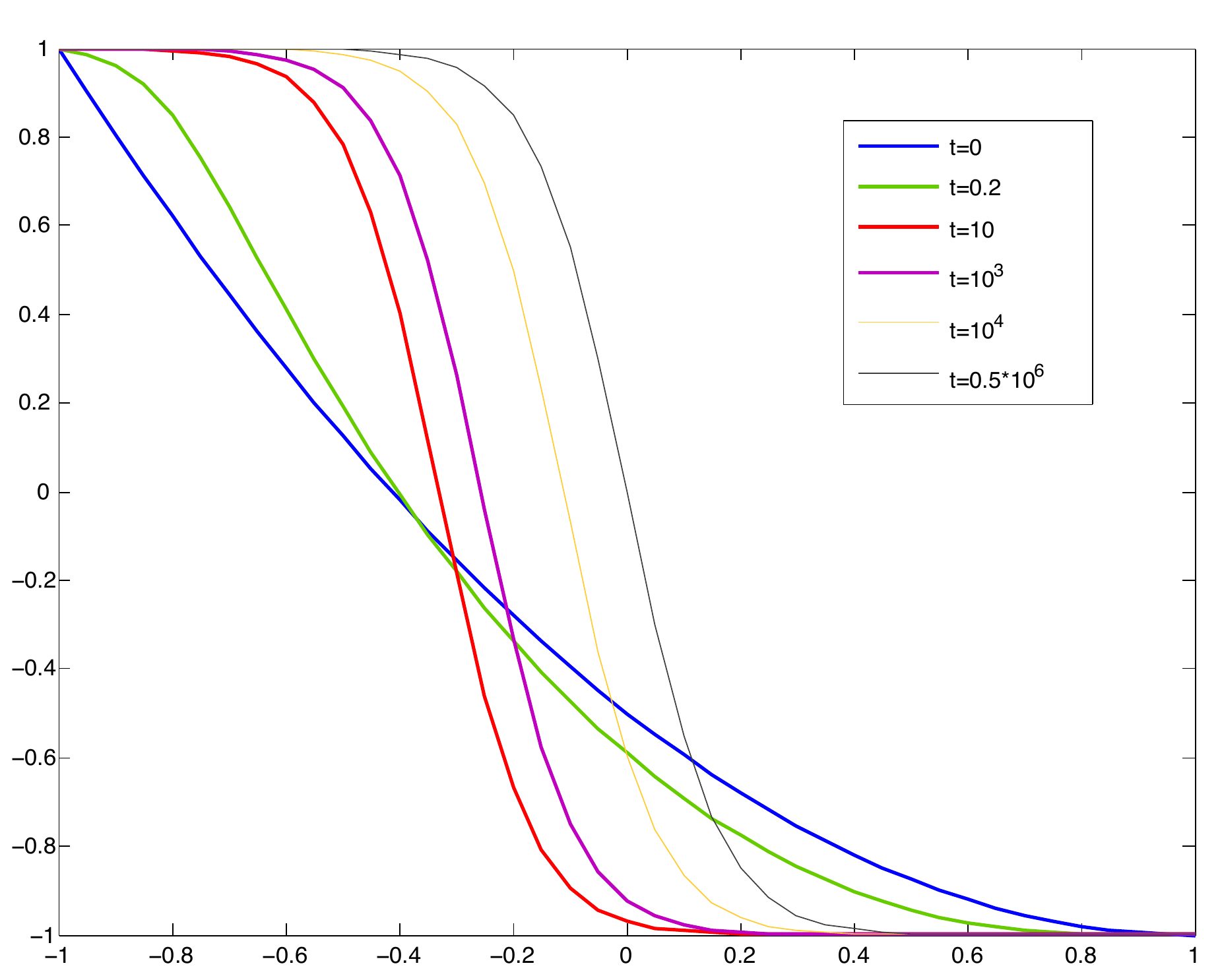}
\caption{\small{The metastable behavior of the solution for the viscous Burgers equation $\partial_t u =\varepsilon \partial_x^2 u - u \partial_x u$ in the interval $I=[-1,1]$ and complemented with Dirichlet boundary conditions $u(\pm 1)=\mp 1$. Starting from a decreasing initial datum, a shock layer is formed in a short time scale; one such a layer appears, it starts to converge towards its asymptotic configuration, corresponding  to an hyperbolic tangent centered in zero, and this motion is extremely slow.} }
\end{figure}

 \begin{figure}
\centering
\includegraphics[width=7cm,height=5cm]{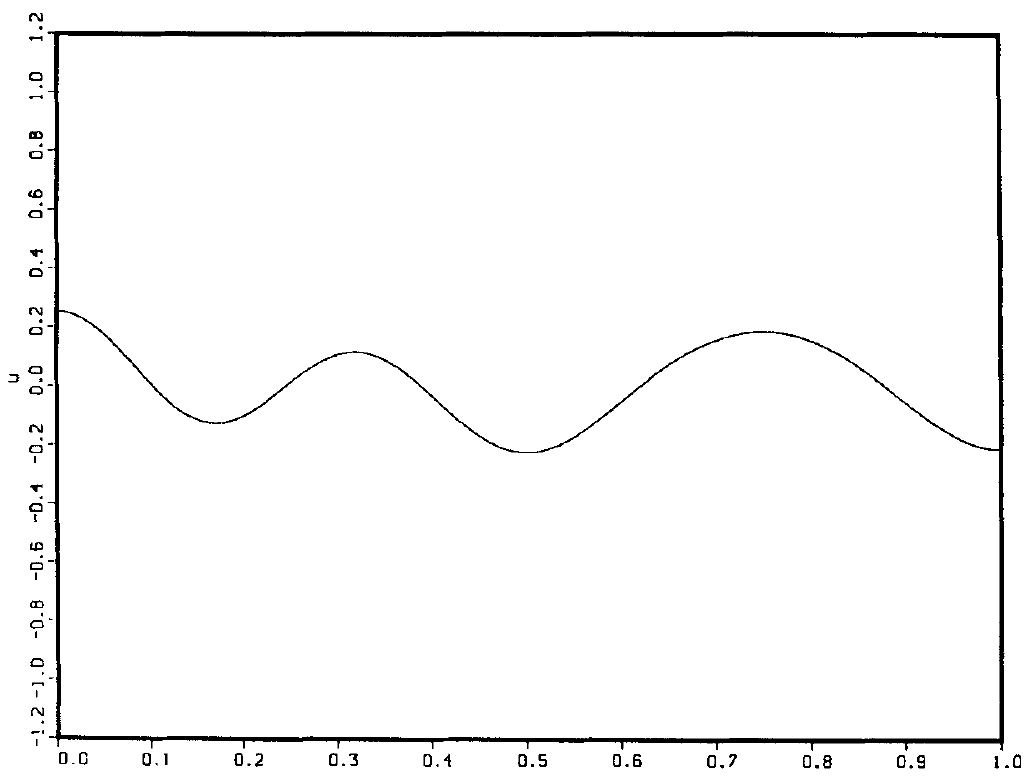}
 \hspace{3mm}
\includegraphics[width=7cm,height=5cm]{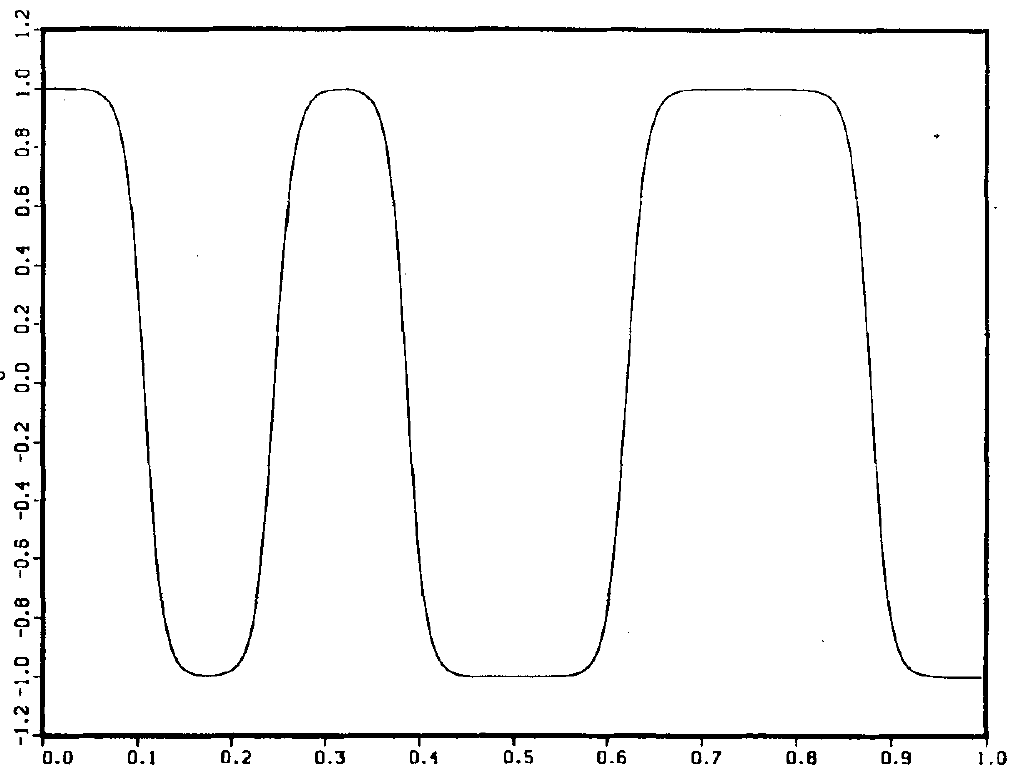}
\caption{\small{The solution to the Allen-Cahn equation $\partial_t{u}=\varepsilon \partial^2_x  { u}-\frac{1}{2}({ u}^3- {u})$ with Neuman boundary conditions $\partial_xu(0)=\partial_xu(1)=0$.
 Starting from an initial datum with ${N}$ zeroes (see picture on the left), a pattern of ${N}$ interfacial layers is formed far out from the equilibrium configuration (given by a patternless constant solution). This pattern persists for an {exponentially long time}, proportional to $e^{1/\varepsilon}.$} This figure was produced by J. Carr and R.L. Pego in \cite{ CarrPego89}. }
\end{figure}

\vskip0.2cm
 In order to describe the dynamics of solutions to \eqref{GS} after the formation of the shock, we follow the strategy firstly developed in \cite{MS}. We summarize the essential steps in order to introduce the objects that will be used in the following.

Given a one-dimensional interval $J$, let $\{U^{\varepsilon}(\cdot;\xi)\,:\,\xi\in J\}$ be a
one-parameter family  of approximate stationary
solutions to \eqref{GS} in the sense that  the nonlinear term ${\mathcal F}^\varepsilon[U^{\varepsilon}(\cdot;\xi)]$
 tends to $0$ as $\varepsilon\to 0$ (we will specify this assumption in details later on).

Next, we decompose the solution to the initial value problem \eqref{GS} as 
\begin{equation*}
	u(\cdot,t)=U^{\varepsilon}(\cdot;\xi(t))+v(\cdot,t),
\end{equation*} 
with $\xi=\xi(t)\in J$  and the perturbation $v=v(\cdot,t)\in [L^2(I)]^n$.
Substituting, one obtains
\begin{equation}\label{eqv}
	\partial_t v=\mathcal{L}^\varepsilon_\xi v+{\mathcal F}^\varepsilon[U^{\varepsilon}(\cdot;\xi)]
		-\partial_{\xi}U^{\varepsilon}(\cdot;\xi)\,\frac{d\xi}{dt}+\mathcal{Q}^\varepsilon[v,\xi]
\end{equation}
where
\begin{equation*}
	\begin{aligned}
		\mathcal{L}^\varepsilon_\xi v&:=d{\mathcal F}^\varepsilon[U^{\varepsilon}(\cdot;\xi)]\,v\\
		\mathcal{Q}^\varepsilon[v,\xi]&:={\mathcal F}^\varepsilon[U^{\varepsilon}(\cdot;\xi)+v]
			-{\mathcal F}^\varepsilon[U^{\varepsilon}(\cdot;\xi)]
				-d{\mathcal F}^\varepsilon[U^{\varepsilon}(\cdot;\xi)]\,v.
	\end{aligned} 
\end{equation*}
Denoting by $\varphi^\varepsilon_k=\varphi^\varepsilon_k(\cdot;\xi)$ and $\psi^\varepsilon_k=\psi^\varepsilon_k(\cdot;\xi)$ the eigenfunctions of $\mathcal L^\varepsilon_\xi$ and of its adjoint respectively, and setting 
\begin{equation*}
	v_k=v_k(\xi;t):=\langle \psi^\varepsilon_k(\cdot;\xi),v(\cdot,t)\rangle, 
\end{equation*}
we impose that 
the component $v_1$  is identically zero, that is
\begin{equation}\label{v1equal0}
	\frac{d}{dt} \langle \psi^\varepsilon_1(\cdot;\xi(t)), v(\cdot,t) \rangle =0
	\qquad\textrm{and}\qquad
	\langle \psi^\varepsilon_1(\cdot;\xi_0), v_0(\cdot))\rangle=0.
\end{equation}
Such a constraint is a consequence of the fact that, as we will see in the following, the first eigenvalue of the linearized operator $\mathcal L^\varepsilon_\xi$ is small in the vanishing viscosity limit, i.e. $\varepsilon \to 0$.
Hence, we set an algebraic condition ensuring orthogonality between ${\psi}^\varepsilon_1$ and $v$, in order to remove the singular part of the operator $\mathcal L^\varepsilon_\xi$.
Precisely, by imposing the first component of the perturbation $v_1$ to be zero, we solve the equation in a subspace in which the operator doesn't vanish.
From \eqref{v1equal0} and using equation \eqref{eqv}, in the regime of small $v$,
we obtain a nonlinear scalar differential equation for the variable $\xi$, describing the reduced
dynamics along the approximate manifold, that is
\begin{equation}\label{eqxiNL}
	\frac{d\xi}{dt}=\theta^\varepsilon(\xi)\bigl(1+\langle\partial_{\xi} \psi^\varepsilon_1, v \rangle\bigr)
		+ \rho^\varepsilon[\xi,v], 
\end{equation}
where
\begin{equation*}
	\begin{aligned}
 	\theta^\varepsilon(\xi)
		&:=\langle \psi^\varepsilon_1,{\mathcal F[U^{\varepsilon}] \rangle}\\
	\rho^\varepsilon[\xi,v]&:=
		\langle \psi^\varepsilon_1,\mathcal{Q}^\varepsilon\rangle
		+\langle \partial_\xi\psi^\varepsilon_1, v\rangle^2.
	\end{aligned}
\end{equation*}
Furthermore, using \eqref{eqxiNL}, equation  \eqref{eqv} can be rewritten as
\begin{equation}\label{eqvNL}
	\partial_t v= H^\varepsilon(x;\xi)
		+ ({\mathcal L}^\varepsilon_\xi+{\mathcal M}^\varepsilon_\xi)v
			+\mathcal{R}^\varepsilon[v,\xi],
\end{equation}
where 
\begin{align*}
		H^\varepsilon(\cdot;\xi)&:={\mathcal F}^\varepsilon[U^{\varepsilon}(\cdot;\xi)]
			-\partial_{\xi}U^{\varepsilon}(\cdot;\xi)\,\theta^\varepsilon(\xi),\\
		{\mathcal M}^\varepsilon_\xi v&:=-\partial_{\xi}U^{\varepsilon}(\cdot;\xi)
			\,\theta^\varepsilon(\xi)\,\langle\partial_{\xi} \psi^\varepsilon_1, v \rangle,\\ 
		\mathcal{R}^\varepsilon[v,\xi]&:=\mathcal{Q}^\varepsilon[v,\xi]
								-\partial_{\xi}U^{\varepsilon}(\cdot;\xi)\,\rho^\varepsilon[\xi,v].
\end{align*}

\section{Estimates for the solution to the  `` quasi linearized " system}\label{Secmeta:line}

Summarizing, the couple $(v,\xi)$ solves the differential system \eqref{eqxiNL}-\eqref{eqvNL}
with initial conditions given by
\begin{equation*}
	\langle \psi^\varepsilon_1(\cdot;\xi_0), u_0-U(\cdot;\xi_0)\rangle =0, \qquad v_0=u_0-U(\cdot;\xi_0).
\end{equation*}
Neglecting the $o(v)$ order terms, we obtain the ``quasi-linearized" system
\begin{equation}\label{LS}
 	\left\{\begin{aligned}
	\frac{d\xi}{dt}&=\theta^\varepsilon(\xi)\bigl(1 
		+\langle\partial_{\xi} \psi^\varepsilon_1, v \rangle\bigr), \\
	\partial_t v &= H^\varepsilon(\xi)+ ({\mathcal L}^\varepsilon_\xi+{\mathcal M}^\varepsilon_\xi)v,
 	\end{aligned}\right. 
\end{equation}
and our aim is to describe the behavior of the solution to \eqref{LS} in the regime
of small $\varepsilon$.
This system is obtained by linearizing with respect to $v$, and by keeping the nonlinear dependence on $\xi$, in order to trace the evolution of the layers far out from their equilibrium configuration.

 In \cite{MS} it has been  proven that the solution $v$ to \eqref{LS} can be decomposed as the sum of two functions, that is $v=z+R$, where $z$ is a function with a very fast decay in time, and the remainder $R$ can be bounded by a term that is small in $\varepsilon$ (for more details, see \cite[theorem 2.1]{MS}). The aim of this Section is to state a slightly different estimate for the perturbation $v$, and to present a new proof to obtain this optimal bound. 

Before state our result, let us introduce the hypotheses we need.
\vskip0.2cm
{\bf H1.} The family $\{U^{\varepsilon}(\cdot,\xi)\}$ is such that there exists smooth functions $\Omega^\varepsilon(\xi)$ such that
\begin{equation*}
	|\langle \psi(\cdot),{\mathcal F}^\varepsilon[U^{\varepsilon}(\cdot,\xi)]\rangle|
		\leq |\Omega^\varepsilon(\xi)|\,|\psi|_{{}_{\infty}} \qquad \forall\,\psi\in C(I)^n,
\end{equation*}
with $\Omega^\varepsilon$ converging to zero as $\varepsilon\to 0$, uniformly with respect to $\xi\in J$. Moreover, we require that there exists a value $\bar \xi \in J$ such that the element $U^\e(x;\bar \xi)$ corresponds to an exact steady state for the original equation.
\vskip0.2cm
{\bf H2.} The eigenvalues $\{\lambda^\varepsilon_k(\xi)\}_{{}_{k\in\N}}$ of the linearized operator $\mathcal{L}^\varepsilon_\xi$ are such that the first eigenvalue is real, negative, and
\begin{equation*}
	\lim_{\varepsilon \to 0}\lambda^\varepsilon_1(\xi) =0, \quad \quad \quad \lambda^\varepsilon_1(\xi)- \Re e\,\lambda^\varepsilon_k(\xi)>c_1 \quad {\rm and} \quad \Re e\,\lambda^\varepsilon_k(\xi) < -c_2
	\quad \textrm{for }k\geq 2.
\end{equation*}
for some constants $c_1,c_2>0$ independent on $k\in\N$, $\varepsilon>0$ and $\xi\in J$.
\vskip0.2cm
{\bf H3.} There exists a constant $C>0$ such that
\begin{equation*}
|\Omega^\varepsilon(\xi)| \leq C |\lambda^\varepsilon_1(\xi)|, \quad \forall \ \xi  \in J.
\end{equation*} 

\vskip0.2cm
{\bf H4.} Concerning the solution $z$ to the linear problem $\partial_t z=\mathcal L^\varepsilon_\xi z$, we require that there exists $\nu^\varepsilon>0$ such that for all $\xi \in J$, there exist a constant $\bar C$ such that
\begin{equation}\label{0H4}
|z(t)|_{{}_{L^2}} \leq \bar C |z_0|_{{}_{L^2}}e^{-\nu^\varepsilon t}, \quad  \forall \xi \in J
\end{equation}
\begin{remark}\rm{
The constant $\bar C$ could depend on $\xi$. In this specific case, since $\xi$ belongs to a bounded interval of the real line, if we suppose that $\xi \mapsto C_{\xi(t)}$ is a continuous function, then there exists a maximum of $C_\xi$  in $J$, namely $\bar C$.
}
\end{remark}

\vskip0.5cm

\begin{theorem}\label{LT}
Let hypotheses {\bf H1-2-3-4} be satisfied. Then, for $\varepsilon$ sufficiently small, the solution $v$ to \eqref{LS} satisfies
\begin{equation*}
|v|_{{}_{L^2}}(t) \leq  C|\Omega^\varepsilon|_{{}_{L^\infty}} t+e^{-\mu^\varepsilon t}|v_0|_{{}_{L^2}}, 
\end{equation*}
for some positive constant $C$ and
\begin{equation*}
\mu^\varepsilon :=\sup_{\xi}\{\lambda^\varepsilon_1(\xi)\}- C|\Omega^\varepsilon|_{{}_{L^\infty}}>0.
\end{equation*}
\end{theorem}

The proof of Theorem \ref{LT} we present here is based on the theory of {\it stable families of generators}, firstly developed by Pazy in \cite{Pazy83}; it is a generalization of the theory of semigroups for evolution systems of the form $\partial_t u= L u$, when the linear operator $L$ depends on time.

Hence, we will first summarize the tools we need to use.

Let us consider the initial value problem 
\begin{equation}\label{pazy}
\partial_t u = A(t) u+ f(t), \quad u(s)=u_0 \qquad 0 \leq s \leq t \leq T.
\end{equation}

\begin{ans}\label{def1}
Let $X$ a Banach space. A family $\{ A(t)\}_{t \in [0,T]}$ of infinitesimal generators of $C_0$ semigroups on $X$ is called stable if there are constants $M \geq 1$ and $\omega$ (called the stability constants) such that
\begin{equation*}
(\omega,+\infty) \subset \rho(A(t)), \quad {\rm for} \quad t \in [0,T]
\end{equation*}
and
\begin{equation*}
\left \| \Pi_{j=1}^k R(\lambda: A(t_j))\right \|  \leq M(\lambda-\omega)^{-k}, \quad {\rm} 
\end{equation*}
for $\lambda>\omega$ and for every finite sequence $0 \leq t_1 \leq t_2, . . . . , t_k \leq T$, $k=1,2,....$.
\end{ans}

If, for $t \in [0,T]$, $A(t)$ is the infinitesimal generator of a $C_0$ semigroup $S_t(s)$, $s \geq 0$ satisfying $\| S_t(s)\| \leq e^{\omega s}$, then the family $\{ A(t)\}_{t \in [0,T]}$ is clearly stable with constants $M=1$ and $\omega$. Precisely, if  the operator $A(t)$ generates a $C_0$ semigroup $S_t(s)$ for every fixed $t \in [0,T]$, and we can find an estimate for $\| S_t(s)\|$ that is independent of $t$, then the whole family $\{ A(t)\}_{t \in [0,T]}$ is stable in the sense of Definition \ref{def1}.

\begin{theorem}\label{thpaz1}
Let $\{ A(t)\}_{t \in [0,T]}$ be a stable family of infinitesimal generators with stability constants $M$ and $\omega$. Let $B(t)$, $0 \leq t \leq T$ be a bounded linear operators on $X$. If $\| B(t)\| \leq K$ for all $t \leq T$, then $\{ A(t)+ B(t)\}_{t \in [0,T]}$ is a stable family of infinitesimal generators with stability constants $M$ and $\omega+ MK$.
\end{theorem}

Now we prove the existence of the so  called {\it evolution system} $U(t,s)$ for the initial value problem \eqref{pazy}, that is a generalization of the semigroup generated by a linear operator $A$, when such operator depends on time. To this aim, let us state the following result (for more details, see \cite[Theorem 2.3, Theorem 3.1, Theorem 4.2]{Pazy83}.

\begin{theorem}\label{thpaz3}
Let $\{ A(t)\}_{t \in [0,T]}$ be a stable family of infinitesimal generators of $C_0$ semigroups on $X$. If $D(A(t))=D$ is independent on $t$ and for $u_0 \in D$, $A(t)u_0$ is continuously differentiable in $X$, then there exists a unique evolution system $U(t,s)$, $0 \leq s \leq t \leq T$, satisfying 
 \begin{equation}\label{01}
 \| U(t,s) \| \leq Me^{\omega (t-s)}, \quad {\rm for } \quad 0 \leq s \leq t \leq T.
 \end{equation}
Morevoer, if $f \in C([s,T],X)$, then, for every $u_0 \in X$, the initial value problem \eqref{pazy} has a unique solution given by 
 \begin{equation}\label{02}
 u(t)= U(t,s)u_0 + \int_s^t U(t,r) f(r) \ dr.
 \end{equation}
for all $0 \leq s \leq t \leq T$.
\end{theorem}

 \vskip0.5cm
\begin{proof}[Proof of Theorem 3.2]
First of all, let us notice that $\mathcal M^\varepsilon_\xi$ is a bounded operator that satisfies the estimate
\begin{equation}\label{asyMH}
\begin{aligned}
\|\mathcal M^\varepsilon_\xi \|_{\mathcal L(L^2;\R)} \leq c_1|\theta^\varepsilon(\xi)| \leq c_1 |\Omega^\varepsilon|_{{}_{L^\infty}}, \quad \forall \xi \in J,
\end{aligned}
\end{equation}
while the term $H^\varepsilon(\xi)$ is such that
\begin{equation}\label{asyMH2}
|H^\varepsilon|_{{}_{L^\infty}} \leq c_2| \Omega^\varepsilon|_{{}_{L^\infty}}.
\end{equation}
For some positive constants $c_1$ and $c_2$. Next, we want to show that $\mathcal L^\varepsilon_\xi+\mathcal M^\varepsilon_\xi$ is the infinitesimal generator of a $C_0$ semigroup $\mathcal T_\xi(t,s)$. To this aim, concerning the eigenvalues of the linear operator $\mathcal L^\varepsilon_\xi$, we know that $\lambda_1^\varepsilon(\xi)$ is negative and goes to zero as $\varepsilon \to 0$, for all $\xi \in J$. Hence, defining $\Lambda_1^\varepsilon:=\sup_\xi \lambda_1^\varepsilon(\xi)$, we have $\lambda_k^\varepsilon \leq -|\Lambda_1^\varepsilon| <0$, for all $k \geq 1$. By using Definition \ref{def1} and the following remark, we know that, for $t \in [0,T]$, $\mathcal L^\varepsilon_{\xi(t)}$ is the infinitesimal generator of a $C_0$ semigroup $\mathcal S_{\xi(t)}(s)$, $s>0$. Furthermore,  since {\bf H4}  holds, we get
\begin{equation*}
\| \mathcal S_{\xi(t)}(s)\| \leq \bar C e^{-|\Lambda_1^\varepsilon|s},
\end{equation*}
and this estimate is independent on $t$, so that the family $\{ \mathcal L^\varepsilon_{\xi(t)}\}_{\xi(t) \in J}$ is stable with stability constants $M=\bar C$ and $\omega=-|\Lambda_1^\varepsilon|$. Furthermore, since
\begin{equation*}
\|\mathcal M^\varepsilon_\xi \|_{\mathcal L(L^2;\R)}  \leq c_1 |\Omega^\varepsilon|_{{}_{L^\infty}}, \quad \forall \xi \in J,
\end{equation*}
Theorem \ref{thpaz1} states that the family $\{ \mathcal L^\varepsilon_{\xi(t)}+ \mathcal M^\varepsilon_{\xi(t)}\}_{\xi(t) \in J}$ is stable with $M=\bar C$ and $\omega= -|\Lambda_1^\varepsilon|+C|\Omega^\varepsilon|_{{}_{L^\infty}} $. In particular, $\omega$ is negative since {\bf H3} holds.

Going further, in order to apply Theorem \ref{thpaz3}, we need to check that the domain of $\mathcal L^\varepsilon_\xi+\mathcal M^\varepsilon_\xi$ does not depend on time;  this is true since $\mathcal L^\varepsilon_\xi+\mathcal M^\varepsilon_\xi$ depends on time through the function $U^\varepsilon(x;\xi(t))$, that does not appear in the higher order terms of the operator. More precisely, the principal part of the operator does not depend on $\xi(t)$.

Hence, we can define $\mathcal T_\xi(t,s)$ as the {\it evolution system} of $\partial_t v=(\mathcal L^\varepsilon_\xi+ \mathcal M^\varepsilon_\xi)v$, so that
\begin{equation}\label{Yfamiglieevol}
v(t)=\mathcal T_\xi(t,s)v_0+\int_s^t \mathcal T_\xi(t,r)H^\varepsilon(x;\xi(r)) dr, \quad 0 \leq s \leq t 
\end{equation} 
Moreover, there holds
\begin{equation*}
\|\mathcal T_\xi(t,s)\| \leq \bar Ce^{-\mu^\varepsilon (t-s)}, \qquad \mu^\varepsilon := |\Lambda^\varepsilon_1|- C|\Omega^\varepsilon|_{{}_{L^\infty}}>0.
\end{equation*}
Finally, from the representation formula \eqref{Yfamiglieevol} with $s=0$, it follows 
\begin{equation}\label{finalestY}
|v|_{{}_{L^2}}(t) \leq e^{-\mu^\varepsilon t}|v_0|_{{}_{L^2}}+ \sup_{\xi \in I}|H^\varepsilon|_{{}_{L^\infty}}(\xi) \int_0^t e^{-\mu^\varepsilon(t-r)} \ dr,  \quad \forall \, t \geq 0,
\end{equation}
so that, by using \eqref{asyMH2}, we end up with
\begin{equation}\label{stimafinaleY}
|v|_{{}_{L^2}}(t) \leq  c_2|\Omega^\varepsilon|_{{}_{L^\infty}}t+e^{-\mu^\varepsilon t}|v_0|_{L^2}.
\end{equation}
and the proof is completed.
\end{proof}

\begin{remark}
\rm{ Let us underline that, with the new proof of \cite[Theorem 2.1]{MS} proposed here, we no longer need the following hypothesis stated in \cite{MS} 
\begin{equation*}
	\sum_{j} \langle \partial_\xi \psi^\varepsilon_k, \varphi^\varepsilon_j\rangle^2
	=\sum_{j} \langle \psi^\varepsilon_k, \partial_\xi \varphi^\varepsilon_j\rangle^2
	\leq C
	\qquad\qquad\forall\,k,
\end{equation*}
and concerning the convergence of the series of the eigenfunctions and their derivatives with respect to $\xi$. Since, at a first step, the approximate system gives a good qualitative picture of the behaviour of the
solution, this is a major improvement, since this hypothesis is not easy to check in specific situations. Also, we end up with an estimate for the perturbation $v$ that, as in \cite{MS}, states that the function $v$ can be decomposed as the sum of a function $z$ such that $|z|_{{}_{L^2}} \leq |v_0|_{{}_{L^2}} e^{-\mu^\varepsilon t}$, plus a remainder that can be bounded by $|\Omega^\varepsilon|_{{}_{L^\infty}}$.

\vskip0.5cm
Estimate \eqref{stimafinaleY} can be used to decouple the system \eqref{LS} in order to obtain an equation of motion for the parameter $\xi(t)$. Indeed, we can state and prove the following consequence of Theorem \ref{LT}.

\begin{corollary}
Let the hypothesis of Theorem \ref{LT} be satisfied. Let us also assume that
\begin{equation*}
	(\xi-\bar \xi)\,\theta^\varepsilon(\xi)<0\quad\textrm{ for any } \xi\in J,\,\xi\neq \bar \xi
	\qquad\textrm{ and }\qquad
	{\theta^\varepsilon}'(\bar \xi)<0.
\end{equation*}
Then, for $\varepsilon$ and $|{v_0}|_{{}_{L^2}}$ sufficiently small, the solution $\xi(t)$ converges exponentially to $\bar \xi$ as $t\to+\infty$.

\end{corollary}

\begin{proof}
Since \eqref{stimafinaleY} holds, we get
\begin{equation*}
\frac{d\xi}{dt}=\theta^\varepsilon(\xi)(1+r) \quad {\rm with} \quad |r| \leq  C|\Omega^\varepsilon|_{{}_{L^\infty}} +e^{-\mu^\varepsilon t}|v_0|_{{}_{L^2}},
\end{equation*}
so that, for $\varepsilon$ and $|v_0|_{{}_{L^2}}$ sufficiently small,
$
\frac{d\xi}{dt}\sim\theta^\varepsilon(\xi).
$
By a method of separation of variables, and since $\theta^\varepsilon(\xi) \sim {\theta^\varepsilon}'(\bar \xi) (\xi -\bar \xi) $, we end up with
\begin{equation*}
\xi(t) \sim \bar \xi+ \xi_0 e^{-\beta^\varepsilon t}, \quad \beta^\varepsilon \sim -  {\theta^\varepsilon}'(\bar \xi),
\end{equation*}
where $\beta^\varepsilon >0$ goes to zero as $\varepsilon \to 0$, and $ \bar\xi$ corresponds to the asymptotic location for the parameter $\xi$ (we recall that $U^\varepsilon(x,\bar \xi)$ is an exact steady state for the system).

\end{proof}
In particular, the motion of the solution towards its asymptotic configuration is much slower as $\varepsilon$ becomes smaller, since the speed rate of convergence is given by $\beta^\varepsilon$. For example, in the case of a single scalar conservation law and for the Allen-Cahn equation, both $\mu^\varepsilon$ and $\beta^\varepsilon$ behave like $e^{-1/\varepsilon}$ (see \cite{MS} and \cite{Str14} respectively), and the convergence of the interfaces towards their equilibrium configuration is indeed  exponentially slow.

}
\end{remark}

\section{Applications to a viscous scalar conservation law and to a relaxation system}

 Our aim in this section is to show how the general approach just presented applies to some explicit examples. In particular, we show that the hypotheses {\bf{ H1-H4}} of Theorem \ref{LT} can be explicitly checked.

 \subsection{Scalar conservation laws}\label{Burg}  We consider the case of the scalar viscous Burgers equation, that is   
 \begin{equation}\label{cauchyBurg1}
		\partial_t u +u \partial_x u	=\varepsilon\,\partial_x^2u, 
		\qquad\qquad x\in I:=(-\ell,\ell),
\end{equation}
with initial and boundary conditions given by
\begin{equation*}
	u(x,0)=u_0(x)\qquad x\in I,
	\qquad\textrm{and}\qquad
	u(\pm\ell,t)=\mp u_\ast \qquad t>0.
\end{equation*}
for some $u_\ast>0$.

The first step is the construction of the family of approximate steady states $\{ U^\varepsilon(x;\xi)\}$; we consider a function obtained by matching two different steady states satisfying, respectively, the left 
and the right boundary condition together with the request $U^\varepsilon(\xi)=0$; in formulas,
\begin{equation}\label{approxU}
	U^{\varepsilon}(x;\xi)=\left\{\begin{aligned}
		&\kappa_- \tanh \left(\kappa_-(\xi-x)/2\varepsilon\right) &\qquad &{\rm in}\quad (-\ell,\xi) \\
		&\kappa_+ \tanh\left(\kappa_+(\xi-x)/2\varepsilon\right) &\qquad &{\rm in}\quad (\xi,\ell),
           \end{aligned}\right.
\end{equation}
where $\kappa_\pm$ are chosen so that the boundary conditions are satisfied.
By direct substitution we obtain the identity
\begin{equation}\label{omega0}
	{\mathcal F}^\varepsilon[U^{\varepsilon}(\cdot;\xi)]
		=\ldbrack \partial_x U^{\varepsilon}\rdbrack_{{}_{x=\xi}}\delta_{{}_{x=\xi}}
\end{equation}
in the sense of distributions, where $\delta_{x=\xi}$ the usual Dirac's delta distribution centered 
in $x=\xi$. Going further, by differentiation, we have
\begin{equation*}
	\ldbrack \partial_x U^{\varepsilon}\rdbrack_{{}_{x=\xi}}
		=\frac{1}{2\varepsilon}(\kappa_--\kappa_+)(\kappa_-+\kappa_+).
\end{equation*}
It is possible to obtain an asymptotic description of the values $\kappa_\pm$ (for more details, see \cite[Section 3]{MS}),
that gives the asymptotic representation
\begin{equation}\label{omega1}
	\begin{aligned}
	\ldbrack \partial_x U^{\varepsilon}\rdbrack_{{}_{x=\xi}}
		&=\frac{2\,u_\ast^2}{\varepsilon}(e^{-u_\ast(\ell+\xi)/\varepsilon}-e^{-u_\ast(\ell-\xi)/\varepsilon})+l.o.t.
			\sim C\,\xi\,e^{-C/\varepsilon},
	\end{aligned}
\end{equation}
showing that the term $\ldbrack \partial_x U^{\varepsilon}\rdbrack_{{}_{x=\xi}}$ is null at $\xi=0$ 
and exponentially small for $\varepsilon\to 0$. In particular, $U^{\varepsilon}(\cdot,\xi)$ is a stationary solution to \eqref{cauchyBurg1} if and only if $\xi=0$, corresponding  to the location of the unique steady state. Finally, gathering \eqref{omega0} and \eqref{omega1}, we have
\begin{equation*}
\Omega^\varepsilon(\xi) \sim \frac{2\,u_\ast^2}{\varepsilon}(e^{-u_\ast(\ell+\xi)/\varepsilon}-e^{-u_\ast(\ell-\xi)/\varepsilon}) =\frac{4\,u_\ast^2}{\varepsilon}|\sinh(u_\ast\,\xi/\varepsilon)|\,e^{-u_\ast \ell/\varepsilon},
\end{equation*}
showing that hypothesis {\bf H1} is satisfied for the viscous Burgers equation.

Next step is to perform a spectral analysis for the linearized operator arising from the linearization of \eqref{cauchyBurg1} around $U^\varepsilon(x;\xi)$. In this specific case, $\mathcal{L}^\varepsilon_\xi $ takes the form 
\begin{equation*}
	\mathcal{L}^\varepsilon_\xi v:=\varepsilon \partial_x^2v-\partial_x \bigl(U^{\varepsilon}(\cdot;\xi)\,v\bigr).
\end{equation*}
Recalling the spectral analysis performed in \cite[Section 4]{MS}, we obtain the following asymptotic expression for the first eigenvalue of the linearized  operator $\mathcal{L}^\varepsilon_\xi$
\begin{equation*}
	\lambda_1^\varepsilon\sim - \frac{u_\ast^2}{2\varepsilon}
		\left(e^{-u_\ast(\ell-\xi)/\varepsilon}+e^{-u_\ast(\ell+\xi)/\varepsilon}\right)
		=-\frac{u_\ast^2}{\varepsilon}\cosh(u_\ast\xi/\varepsilon)\,e^{-u_\ast\ell/\varepsilon},
\end{equation*}
to be compared with the expression for $\Omega^\varepsilon$, so that
\begin{equation*}
	\left|\frac{\Omega^\varepsilon}{\lambda_1^\varepsilon}\right|\sim 4|\tanh(u_\ast\,\xi/\varepsilon)|\leq 4.
\end{equation*}
This formula shows that hypothesis {\bf H3} is verified for Burgers type equations. 

Concerning the eigenvalues greater or equal to the second, there holds (see \cite[Corollary 4.5]{MS})
\begin{equation*}
\lambda_k^\varepsilon(\xi) \leq -\frac{c}{\varepsilon}, \qquad \forall \ k \geq 2,
\end{equation*}
showing that the spectral distribution required in hypothesis {\bf H2} is in this case satisfied. Finally, hypothesis {\bf H4} is a direct consequence of the negativity of the first eigenvalue.

\vskip0.2cm
The rigorous results are also validated numerically; the following table shows a numerical computation for the location of the shock layer  for different values of the parameter $\varepsilon$. The initial datum is the nondecreasing function $u_0(x)=\frac{1}{2}x^2-x-\frac{1}{2}$, and the equation is considered in the bounded interval $[-1,1]$, and complemented with Dirichlet boundary conditions $u(\pm 1)= \mp 1$. In this case, as already stressed, the single stable steady state is given by the hyperbolic tangent centered in zero  $U^\varepsilon(x;0)=-\kappa \tanh \left( \frac{\kappa x}{2\varepsilon}\right)$, where $\kappa=\kappa(\varepsilon, u_\pm)$. We can see that the convergence towards $\bar \xi =0$  is slower as $\varepsilon$ becomes smaller.

\begin{table}[h!]
{\footnotesize{The numerical location of the shock layer $\xi(t)$ for different values of the parameter $\varepsilon$}}
\begin{center}
\begin{tabular}{|c|c|c|c|c|c|}
\hline TIME $t$ &   $\xi(t)$, $\varepsilon=0.1$  &   $\xi(t)$, $\varepsilon=0.07$ &  $\xi(t)$, $\varepsilon=0.06$ &   $\xi(t)$, $\varepsilon=0.04$ 
\\ \hline
\hline $0.2$ Ê& $-0.3954$ &$-0.4010$  &$-0.4028$ &$-0.4065$ \\
\hline  $1$ & $-0.3233$ & $-0.3293$  &$-0.3306$ &$-0.3808$ \\
\hline $5*10^3$ & $-0.0044$ &$-0.2231$  &$-0.3032$ &$-0.3315$ \\
\hline $5*10^4$ & $-4.3033*10^{-12}$ & $-0.0942$ & $-0.2198$ & $-0.3314$\\
\hline $10^5$ & $-4.3033*10^{-12}$ &$-0.0528$ &$-0.1845$ &$-0.2531$\\
\hline $10^6$ & $-4.3033*10^{-12}$ &$-8.7386*10^{-6}$ &$-8.7386*10^{-6}$ &$-0.0379$ \\
\hline\end{tabular}
\end{center}

\end{table}

Figure \ref{fig2} shows the dynamics of the shock layer, obtained numerically. When $\varepsilon=0.1$, the shock layer location converges to zero very fast: as we can also see from the table, when $t=5*10^3$, the value of $\xi(t)$ is already very close to zero, corresponding to its equilibrium. On the other hand, when $\varepsilon$ becomes smaller the shock layer location moves slower and it approaches the equilibrium location only for very large $t$.

\begin{figure}
\centering
\includegraphics[width=7cm,height=7cm]{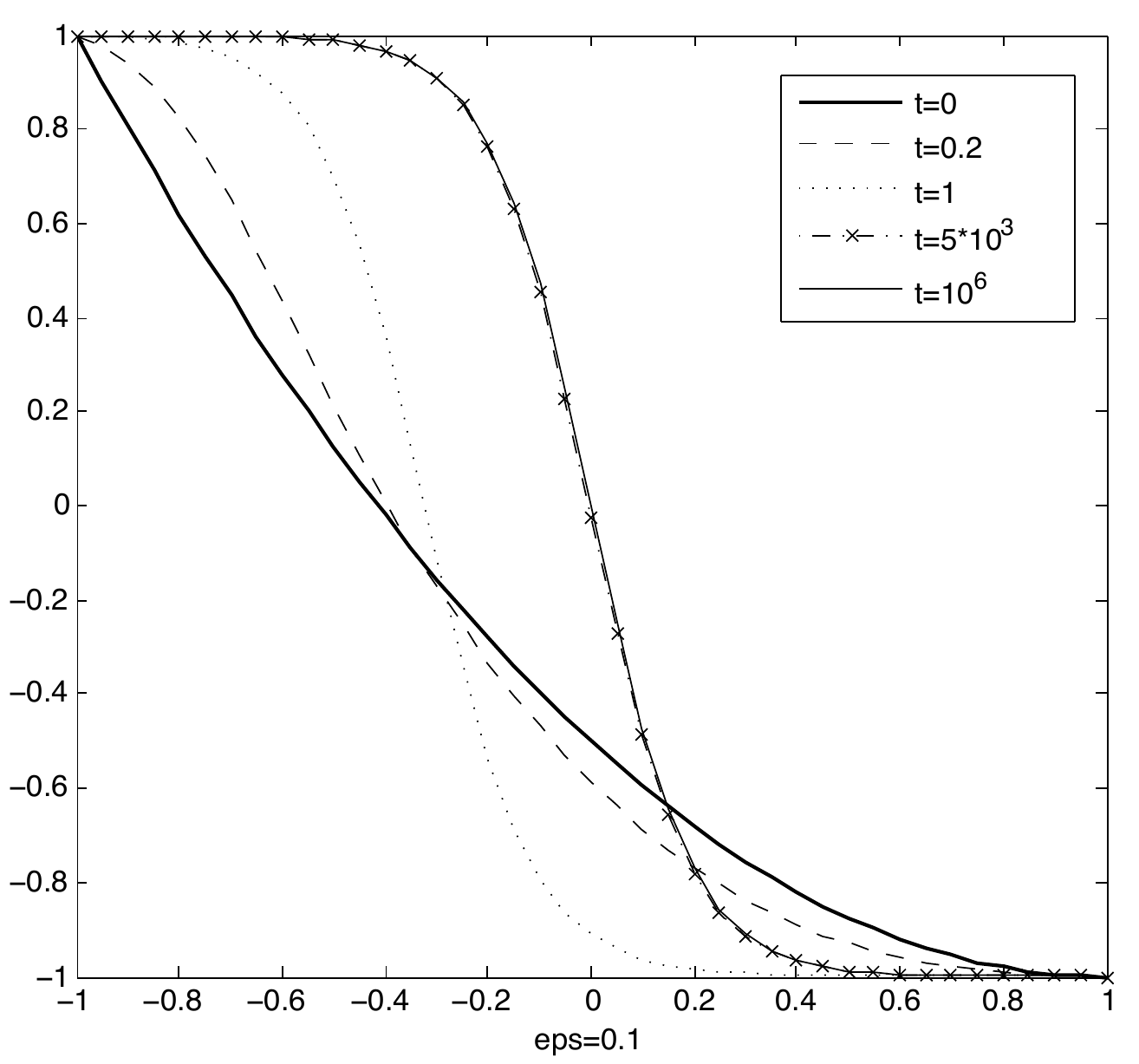}
 \hspace{3mm}
\includegraphics[width=7cm,height=7cm]{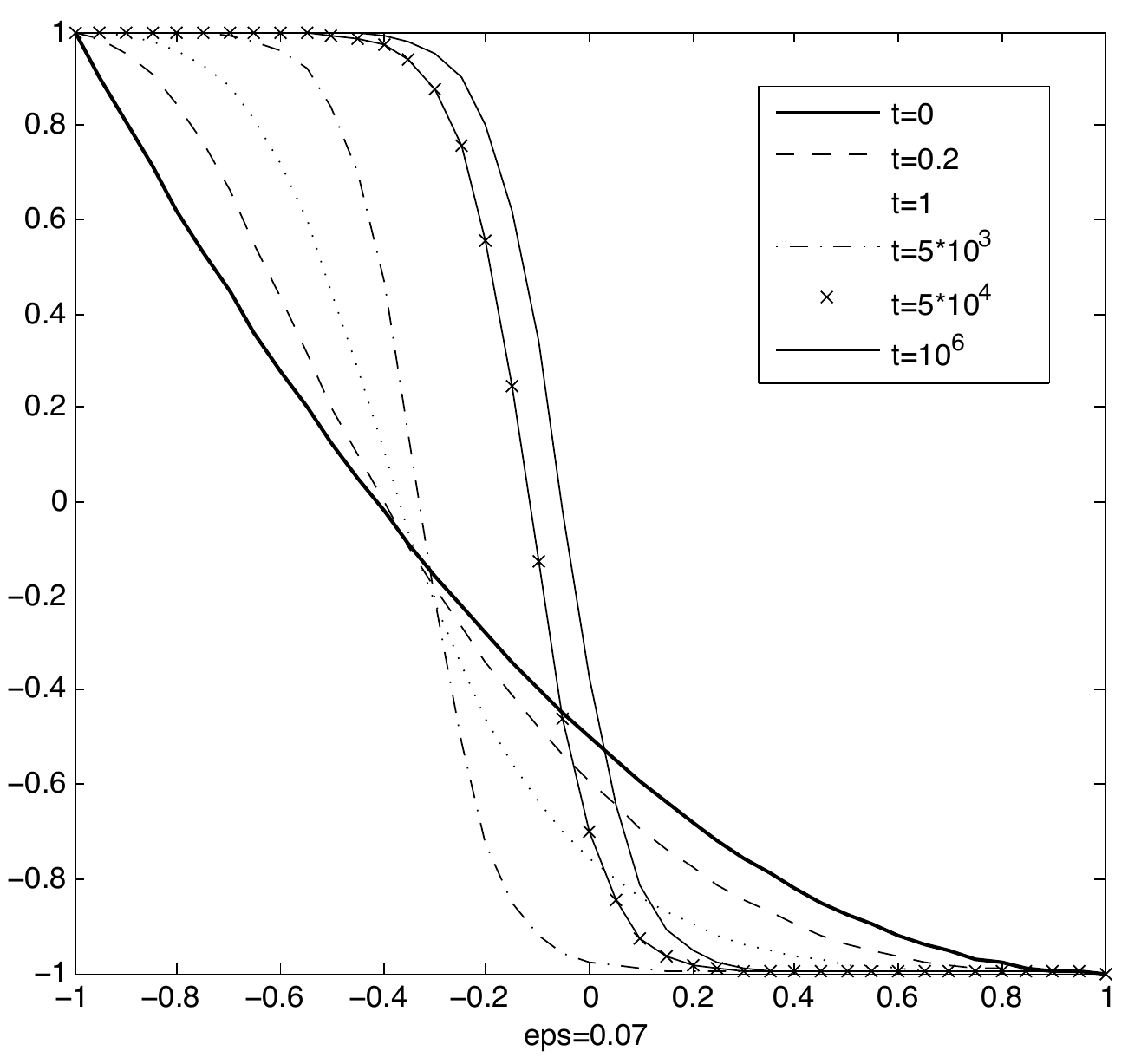}
 \hspace{3mm}
\includegraphics[width=7cm,height=7cm]{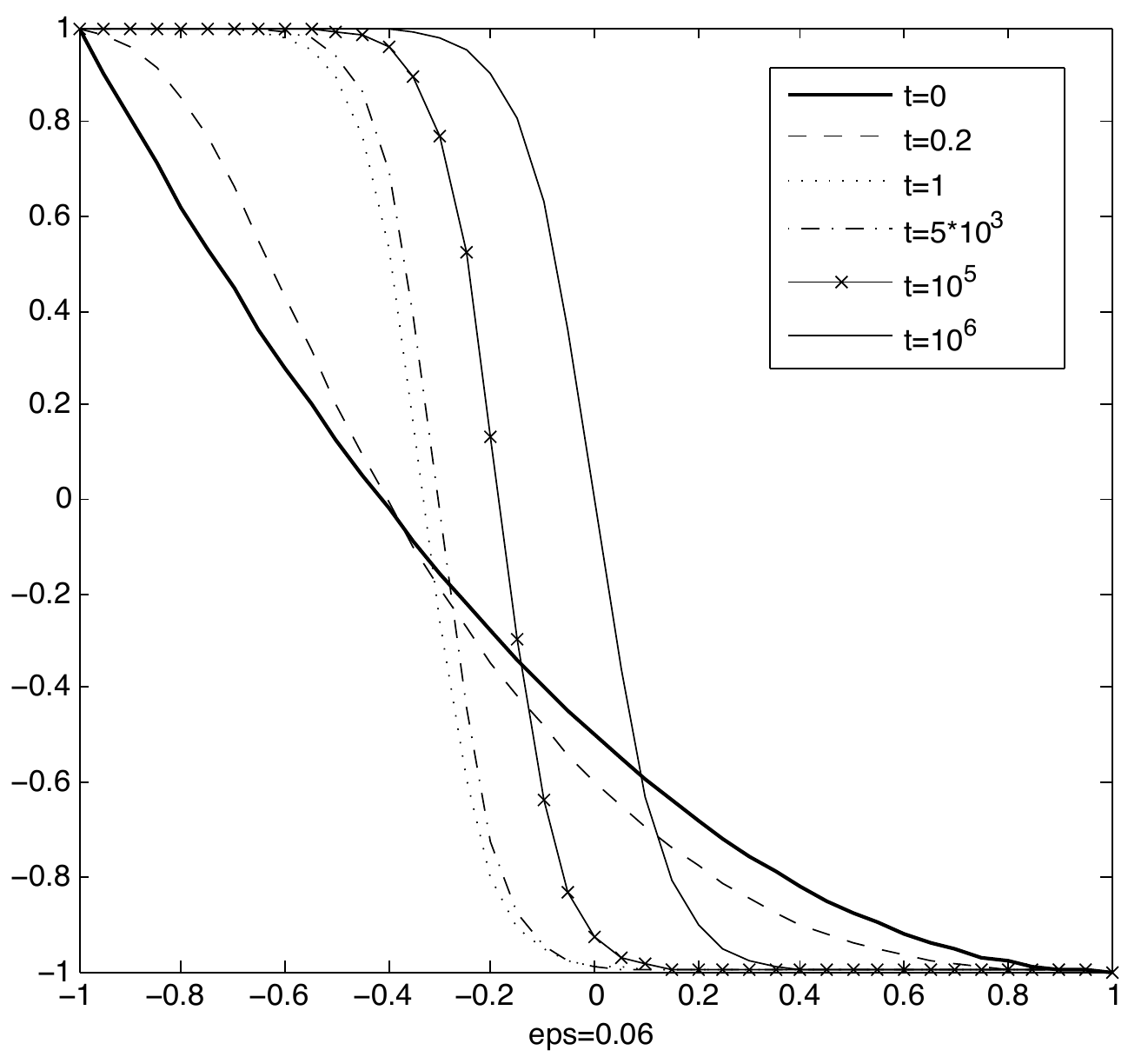}
 \hspace{3mm}
\includegraphics[width=7cm,height=7cm]{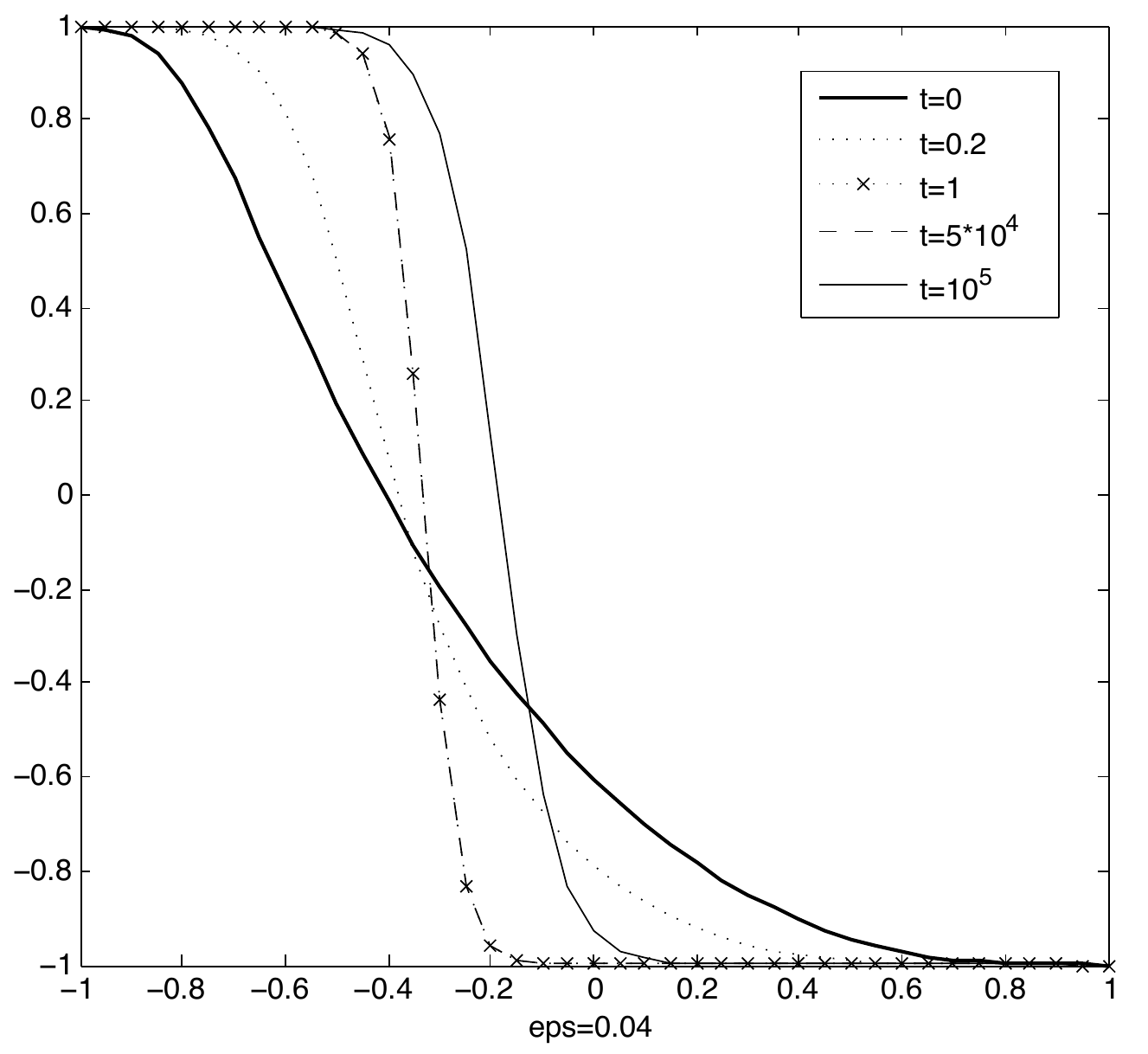}
\caption{\small{The shock layer profiles for the viscous Burgers equation for different times and different values of the parameter $\varepsilon$. }}\label{fig2}
\end{figure}

\subsection{The hyperbolic-parabolic Jin-Xin system}\label{JX}

We consider  the initial-boundary value problem for the quasilinear Jin-Xin system in the bounded interval $I=[-\ell,\ell]$  with Dirichlet boundary conditions, that is
\begin{equation}\label{JXbur}
 \left\{\begin{aligned}
& \partial_t u +\partial_x v=0,  &\qquad &x \in I, \ t \geq 0, \\
&\partial_t v+ a^2 \partial_x u=\frac{1}{\varepsilon} (f(u)-v), \\
& u(\pm \ell,t)=u_{\pm}, &\qquad &t \geq 0,\\
& u(x,0)=u_0(x), \quad v(x,0)=v_0(x) \equiv f(u_0(x)), &\qquad &x \in I,
  \end{aligned}\right.
\end{equation}
for some $\varepsilon, \ell,a>0$, and $u_{\pm} \in \R$. To simplify the computations, we consider the specific case of the quadratic flux function $f(u)=u^2/2$, and $a=1$.
\vskip0.2cm
The Jin-Xin model was firstly introduced in \cite{JinXin95} as a numerical scheme approximating the solutions of the hyperbolic conservation law $\partial_t u+ \partial_xf(u)=0$. In particular, the study of stationary solutions to \eqref{JXbur}  is exactly the same of that of the scalar conservation law \eqref{cauchyBurg1},
together with the additional condition $\partial_x v=0$. 

Hence, in order to build up the family of approximate steady states $\{ \boldsymbol{U}^{\varepsilon}(\cdot,\xi)\}= \{ (U^{\varepsilon}, V^\varepsilon)(\cdot,\xi)\}$, we proceed as follows: stationary solution to \eqref{JXbur} satisfies
\begin{equation*}
\varepsilon \partial_x u= \frac{u^2}{2}-\frac{c^2}{2}, \quad v=\frac{c^2}{2}, 
\end{equation*}
with boundary conditions $u(\pm \ell)=\mp u^*$, for some $u^*>0$. Hence, the family $U^\varepsilon$ can be constructed as in the previous section, via de formula \eqref{approxU}.  Moreover, by the condition $v=\frac{c^2}{2}$, we have
\begin{equation*}
V^\varepsilon(x;\xi)= \left\{\begin{aligned}
&\kappa^2_- /2 \quad \rm in \ (-\ell,\xi), \\
&\kappa^2_+/2 \quad \rm in \ (\xi,\ell).
  \end{aligned}\right.
\end{equation*}
In order to obtain an asymptotic expression for the term $\boldsymbol{\Omega}^\varepsilon(\xi)= (\Omega_1^\varepsilon(\xi), \Omega_2^\varepsilon(\xi))$, we compute
\begin{equation*}
\mathcal F^\varepsilon[\boldsymbol{U}^\varepsilon]:= \left( \begin{aligned} &-\partial_x V^\varepsilon \\ -\partial_x U^\varepsilon+&\frac{1}{\varepsilon}\left\{ (U^{\varepsilon})^2/2\!-\!V^\varepsilon\right\} \end{aligned}\right).
\end{equation*}
From the explicit formula for $U^\varepsilon(x;\xi)$ given in \eqref{approxU}, and since  $-\partial_x V^\varepsilon=\varepsilon \partial_x^2U^\varepsilon-U^\varepsilon \partial_x U^\varepsilon$, by direct substitution we obtain the identity
\begin{equation*}
\mathcal F^\varepsilon[\boldsymbol{U}^\varepsilon]= (0, [\![\partial_xU^\varepsilon]\!]_{x=\xi}\delta_{x=\xi})^{t}.
\end{equation*}
Recalling the definition of $\boldsymbol{\Omega}^\varepsilon$ we have
\begin{equation*}
\boldsymbol{\Omega}^\varepsilon(\xi) \sim (0, \frac{2 u^2_*}{\varepsilon}(e^{-u_*(\ell+\xi)/\varepsilon}-e^{-u_*(\ell-\xi)/\varepsilon})^t,
\end{equation*}
showing that hypothesis {\bf H1} is satisfied in the case of the Jiin-Xin system. For more details on the computations, see \cite[Example 2.2]{Str12}.

Concerning the spectral analysis, we will makes use of the analogies between this problem and the viscous scalar conservation law \eqref{cauchyBurg1}.  By linearizing the system \eqref{JXbur} around $\{ (U^\varepsilon, V^\varepsilon) \}$, the eigenvalue problem  for  the linearized operator $\mathcal L^\varepsilon_\xi$ reads
\begin{equation*}
\left \{ \begin{aligned}
\lambda \varphi &=-\partial_x \psi,\\
\lambda \psi &= - \partial_x \varphi+ \frac{1}{\varepsilon} (U^{\varepsilon} \varphi-\psi).
\end{aligned} \right.
\end{equation*}
By differentiating the second equation with respect to $x$, we obtain
\begin{equation*}
\varepsilon \partial_x^2 \varphi-\partial_x(U^\varepsilon \varphi)=\lambda(1+\varepsilon \lambda)\varphi. 
\end{equation*}
Hence, $\lambda$ is an eigenvalue for $\mathcal L^\varepsilon_\xi$ if and only if $\lambda^{vsc}:=\lambda(1+\varepsilon \lambda)$ is an eigenvalue for the operator $\mathcal L^{\varepsilon,vsc}$ defined as
\begin{equation*}
\mathcal L^{\varepsilon,vsc} \varphi := \varepsilon  \partial_x^2 \varphi-\partial_x (U^\varepsilon \varphi).\end{equation*}
Thus, by using the spectral analysis performed in section \ref{Burg} for the linear operator $\mathcal L^{\varepsilon,vsc}$, one can prove the following result (for more details, see \cite[Proposition 3.4]{Str12}); for the first eigenvalue of the linearized operator $\mathcal L^\varepsilon_\xi$ there holds $\lambda_1^{JX}(\xi)$ negative for all $\xi$ and 
\begin{equation}\label{lambdaasybur}
|\lambda_{1}^{JX}(\xi)| \sim \frac{ \frac{{u^*}^2}{\varepsilon}\left[  e^{-u^*\varepsilon^{-1}(\ell-\xi)}+e^{-u^* \varepsilon^{-1}(\ell+\xi)}\right]}{1+ \sqrt{1-2  {u^*}^2 \left[  e^{-u^*\varepsilon^{-1}(\ell-\xi)}+e^{-u^*\varepsilon^{-1}(\ell+\xi)}\right]}}.
\end{equation}
Moreover, all the other eigenvalues are bounded away from zero and such that
\begin{equation*}
\Re e \, \lambda_k^\varepsilon(\xi) \leq -\frac{c}{\varepsilon}, \qquad \forall \ k \geq 2.
\end{equation*}
Thanks to these results, and by comparing \eqref{lambdaasybur} with the expression obtained for $\boldsymbol{\Omega}^\varepsilon$, we can easily check that hypotheses {\bf H2-3-4} are satisfied for the Jin-Xin system.

\vskip0.2cm
Again, we can give evidence of the rigorous theory by numerically compute the solution to \eqref{JXbur}: Figure 4  shows that a shock layer is formed from initial data in a $\mathcal O(1)$ time scale.  Once this interface is formed, it moves towards the equilibrium solution, and this motion is exponentially slow. Concerning the function $v$, starting with the initial datum $v_0(x)=f(u_0(x))$, we can observe that the position of the shock of $u$ corresponds to the location of the minimum value of the function $v$.

\begin{figure}[ht]
\centering
\includegraphics[width=12cm,height=10cm]{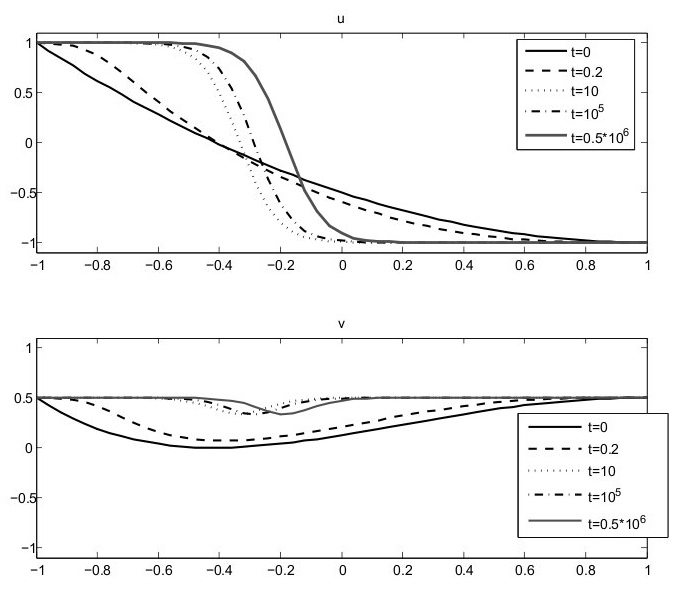}

\caption{\small{Profiles of $(u,v)$, solutions to \eqref{JXbur}, with $f(u)=u^2/2$, a =1 $\varepsilon=0.04$ and $u_{\pm}=\mp1$. The initial datum is given by the couple $(u_0(x),f(u_0(x)))$, with $u_0(x)$ a decreasing function connecting $u_+$ and $u_-$.  }}\label{fig1}
\end{figure}

The following table shows a numerical computation for the location of the shock layer (corresponding to the zero value of the function $u$) for different values of the parameter $\varepsilon$. The convergence towards the equilibrium is slower as $\varepsilon$ becomes smaller; for example, when $\varepsilon= 0.02$, the interface is almost still.

\begin{table}[h!]
\footnotesize{The numerical location of the shock layer $\xi(t)$ for different values of the parameter $\varepsilon$}
\begin{center}
\begin{tabular}{|c|c|c|c|c|c|}
\hline TIME $t$ &   $\xi(t)$, $\varepsilon=0.1$  &   $\xi(t)$, $\varepsilon=0.07$ &  $\xi(t)$, $\varepsilon=0.055$ &   $\xi(t)$, $\varepsilon=0.04$ &  $\xi(t)$, $\varepsilon=0.02$
\\ \hline
\hline $0.2$ Ê& $-0.4008$ &$-0.4020$  &$-0.4029$ &$-0.4040$ &$-0.4059$\\
\hline  $1$ & $-0.3314$ & $-0.3345$  &$-0.3360$ &$-0.3374$ & $-0.3389$\\
\hline $10$ & $-0.3070$ &$-0.3263$ &$-0.3304$  &$-0.3320$ & $-0.3326$\\
\hline $10^3$ & $-0.0103$ &$-0.1600$  &$-0.2562$ &$-0.3181$ & $-0.3325$\\
\hline $10^4$ & $-1.9725*10^{-12}$ &$-0.0084$ &$-0.1115$ &$-0.2531$& $-0.3320$\\
\hline $0.5*10^6$ & $-1.9725*10^{-12}$ &$-2.2102*10^{-11}$ &$-1.5057*10^{-10}$ &$-0.0379$ & $-0.3099$\\
\hline\end{tabular}
\end{center}

\end{table}

\begin{remark}{\rm
The approach presented in the previous sections  is a general approach that can be used for a broad class of parabolic system, providing to be able to perform a spectral analysis of the equation under consideration. In principle, this analysis may be done also numerically.  We quote the papers \cite{Str13,Str14}, where the  Burgers-Sivashinsky and the Allen-Cahn equations are studied by using this technique, with the appropriate changes due to the specific form of the equations under consideration. }
\end{remark}

\section{Nonlinear metastability for parabolic systems of conservation laws}

The aim of this Section is to study the behavior of the solution $(\xi,v)$ to the complete system 
\begin{equation}\label{CS}
 	\left\{\begin{aligned}
	\frac{d\xi}{dt}&=\theta^\varepsilon(\xi)\bigl(1 
		+\langle\partial_{\xi} \psi^\varepsilon_1, v \rangle\bigr) + \rho^\varepsilon[\xi,v], \\
	\partial_t v &= H^\varepsilon(\xi)+ ({\mathcal L}^\varepsilon_\xi+{\mathcal M}^\varepsilon_\xi)v +\mathcal{R}^\varepsilon[v,\xi],
 	\end{aligned}\right. 
\end{equation}
where also the higher order terms arising from the linearization around $v \sim 0$ are taken into account. As already stressed in the Introduction, dealing with the nonlinear terms in the perturbation $v$ brings into the analysis the specific form of the quadratic terms: in some cases, like the one of parabolic systems of reaction-diffusion equations, these terms do not depend on the space derivatives of the solutions, and a result analogous to Theorem \ref{LT} can be proved without additional assumptions (see \cite{Str14}).
\vskip0.2cm
Differently, when a nonlinear first-order space derivative term is present, as is the case of viscous conservation laws, the quadratic term involves a dependence on the space derivative of the solution, and an additional bound for the $H^1-$norm is needed.

Here we reduce our analysis to the specific case of systems of viscous scalar conservation laws; a straightforward computation shows that the nonlinear term $\mathcal Q^\varepsilon[\xi,v]$ is  given by $ \mathcal Q^\varepsilon:=v\partial_xv $. Precisely, $|\mathcal Q^\varepsilon| \leq C |v|^2_{{}_{H^1}}$, and this estimate can be used to prove the following result.

\begin{theorem}\label{NT}
 Let us denote by $(\xi,v)$ the solution to the initial-value problem \eqref{CS} with
\begin{equation*}
\xi(0)=\xi_0  \in J \quad {\rm and} \quad v(x,0)=v_0(x) \in  H^1(I),
\end{equation*}
Let  hypotheses {\bf H1-2-3} be satisfied and let us also assume that the eigenfunctions $\varphi^\varepsilon_k(\cdot;\xi)$ and $\psi^\varepsilon_k(\cdot;\xi)$ of $\mathcal{L}^\varepsilon_{\xi}$ and $\mathcal{L}^{\varepsilon,\ast}_{\xi}$ are such that
\begin{equation}\label{derpsiphi}
	\sum_{j} \langle \partial_\xi \psi^\varepsilon_k, \varphi^\varepsilon_j\rangle^2
	=\sum_{j} \langle \psi^\varepsilon_k, \partial_\xi \varphi^\varepsilon_j\rangle^2
	\leq C
	\qquad\qquad\forall\,k,
\end{equation}
for some constant $C$ independent on $\varepsilon>0$ and $\xi\in J$.

Then,  for $|v_0|_{{}_{H^1}}$ and $\varepsilon$ sufficiently small,  the solution $v$ can be decomposed as
$
v=z+R,
$
where $z$ is defined by
\begin{equation*}
	z(x,t):=\sum_{k\geq 2} v_k(0)\exp\left(\int_0^t \lambda^\varepsilon_k(\xi(\sigma))\,d\sigma\right)
		\,\varphi^\varepsilon_k(x;\xi(t)),
\end{equation*}
and the remainder $R$ satisfies the estimate
\begin{equation}\label{boundresto}
	|R|_{{}_{H^1}}\,\leq C\,
		\left\{ \varepsilon^\delta \, \exp\left( \int_0^t\lambda_1^\varepsilon(\xi(\sigma))d\sigma \right)|v_0|^2_{{}_{H^1}}+\varepsilon^{1-\delta}\right\}
\end{equation}
for some constant $C>0$ and for some $\delta \in (0,1)$. 
\end{theorem}

We point out that we imposed again the additional hypothesis \eqref{derpsiphi}. This is probably related to the specific strategy we use at such stage, but it is needed in order to provide an estimate for the $H^1$- norm of the perturbation $v$, making the theory more complete.

Also, estimate \eqref{boundresto} is weaker that the corresponding estimate \eqref{stimafinaleY} obtained for the reduced system in Section \ref{Secmeta:line}, since it states that the remainder $R$ tends to $0$ as $\varepsilon^\delta$ instead of $|\Omega^\e|_{{}_{\infty}}$. Such deterioration  is a consequence of the necessity of estimating also the first order derivative.

\begin{proof}
Since the plan of the proof closely resembles the one used in \cite[Theorem 2.1]{MS}, we
propose here only the major modifications of the argument. Especially, the key point is how to deal with the nonlinear higher order terms.

Setting 
\begin{equation*}
	v(x,t)=\sum_{j} v_j(t)\,\varphi^\varepsilon_j(x,\xi(t)),
\end{equation*}
we obtain an infinite-dimensional differential system for the coefficients $v_j$
\begin{equation}\label{eqwk_bis}
	\frac{dv_k}{dt}=\lambda^\varepsilon_k(\xi)\,v_k
		+\langle \psi^\varepsilon_k,F\rangle+ +\langle \psi^\varepsilon_k,G\rangle,
\end{equation}
where $F$ is defined as 
\begin{equation*}
	F:=H^\varepsilon-\theta^\varepsilon \sum_{j}\Bigl(a_j+\sum_{\ell} b_{j\ell}\,v_\ell\Bigr)v_j,
\end{equation*}
with
\begin{equation*}
	a_j:=\langle \partial_{\xi} \psi^\varepsilon_1, \varphi^\varepsilon_j\rangle\,
			\partial_{\xi}U^{\varepsilon}+\partial_\xi \varphi^\varepsilon_j,
	\qquad
	b_{j\ell}:=\langle \partial_{\xi} \psi^\varepsilon_1, \varphi^\varepsilon_\ell\rangle
			\,\partial_\xi \varphi^\varepsilon_j.
\end{equation*}
The term $G$ comes out from the higher order terms $\rho^\varepsilon$ and $\mathcal R^\varepsilon$ and has the following expression
\begin{equation*}
G:= \mathcal Q^\varepsilon- \left(  \sum_j  \partial_\xi \varphi_j^\varepsilon v_j+\partial_\xi U^\varepsilon\right) \left\{ \frac{\langle \psi^\varepsilon_1,\mathcal Q^\varepsilon \rangle}{1-\langle\partial_\xi \psi^\varepsilon_1,v \rangle}-\theta^\varepsilon\frac{\langle \partial_\xi \psi^\varepsilon_1,v \rangle^2}{1-\langle\partial_\xi \psi^\varepsilon_1,v \rangle} \right\}.
\end{equation*}
Moreover, we have
\begin{equation*}
|\langle \psi^\varepsilon_k, G \rangle| \leq (1+|\Omega^\varepsilon|_{{}_{L^\infty}}) |v|^2_{{}_{L^2}}+ C |v|^2_{{}_{H^1}},
\end{equation*}
so that, setting
\begin{equation*}
	E_k(s,t):=\exp\left( \int_s^t \lambda_k^\varepsilon(\xi(\sigma))d\sigma\right)
\end{equation*}
and since $v_1=0$, we have the following expression for the coefficients $v_k$, $k\geq 2$
\begin{equation*}
v_k(t)= v_k(0) E_k(0,t)+ \int_0^t \left\{ \right \langle \psi^\varepsilon_k,F\rangle+ +\langle \psi^\varepsilon_k,G\rangle \} E_k(s,t) \, ds.
\end{equation*}
By introducing the function
\begin{equation*}
	z(x,t):=\sum_{k\geq 2} v_k(0)\,E_k(0,t)\,\varphi^\varepsilon_k(x;\xi(t)),
\end{equation*}
we end up with the following estimate for the $L^2$-norm of the difference $v-z$
\begin{equation*}
|v-z|_{{}_{L^2}} \leq \sum_{k \geq 2} \int_0^t \left( \Omega^\e(\xi)(1+|v|^2_{{}_{L^2}}) + |v|^2_{{}_{H^1}}\right) E_k(s,t) \, ds,
\end{equation*}
that is
\begin{equation*}
\begin{aligned}
|v-z|_{{}_{L^2}} \leq& \ C \int_0^t \Omega^\varepsilon(\xi) (t-s)^{-1/2} E_2(s,t)  ds \\
&+ \int_0^t \left( |v-z|^2_{{}_{H^1}}+ |z|^2_{{}_{H^1}}\right) (t-s)^{-1/2} E_2(s,t)  ds,
\end{aligned}
\end{equation*}
where we used
 \begin{equation*}
	\sum_{k\geq 2} E_k(s,t)
		\leq C\,(t-s)^{-1/2}\,E_2(s,t).
\end{equation*}
Now we need to differentiate with respect to $x$ the equation for $v$ in order to obtain an estimate for $|\partial_x(v-z)|_{{}_{L^2}}$. By setting $y=\partial_x v$, we obtain
\begin{equation*}
\partial_t y= \mathcal L^\varepsilon_\xi y + \bar{\mathcal  M}^\varepsilon_\xi v-\partial_x\left( \partial_xU^\varepsilon \, v\right) + \partial_xH^\varepsilon(x,\xi)+\partial_x \mathcal R^\varepsilon[\xi,v],
\end{equation*}
where
\begin{equation*}
\bar{\mathcal M}^\varepsilon_\xi v:=-\partial_{\xi x}U^{\varepsilon}(\cdot;\xi)
			\,\theta^\varepsilon(\xi)\,\langle\partial_{\xi} \psi^\varepsilon_1, v \rangle.
\end{equation*}
Hence, by setting as usual
\begin{equation*}
y(x,t)=\sum_j y_j(t) \, \varphi_j^\varepsilon (x,\xi(t)),
\end{equation*}
we have
\begin{equation*}
\frac{dy_k}{dt}=\lambda^\varepsilon_k(\xi)\,y_k
		+\langle \psi^\varepsilon_k,F^*\rangle-\langle \psi^\varepsilon_k ,\partial_x\left( \partial_x U^\varepsilon\,v\right)\rangle +\langle \psi^\varepsilon_k,\partial_x \mathcal R^\varepsilon \rangle,
\end{equation*}
where
\begin{equation*}
F^*:= \partial_x H^\varepsilon -\sum_j v_j \left \{ \theta^\varepsilon\left[ \partial_{\xi x} U^\varepsilon \langle \partial_\xi \psi^\varepsilon_1,\varphi_j^\varepsilon \rangle + \partial_\xi \varphi^\varepsilon_j\left( 1 + \sum_{\ell}v_{\ell} \langle \partial_\xi \psi^\varepsilon_1,\varphi_{\ell} \rangle \right)   \right] -\partial_\xi \varphi^\varepsilon_j \rho^\varepsilon\right \}.
\end{equation*}
Moreover, for some $m >0$, there hold
\begin{equation*}
|\langle \psi^\varepsilon_k, \partial_x \mathcal R^\varepsilon \rangle| \leq C |v|^2_{{}_{H^1}}, \quad |\langle \psi^\varepsilon_k ,\partial_x\left( \partial_x U^\varepsilon \, v\right)\rangle| \leq \varepsilon^m\, |U^\varepsilon|_{{}_{L^\infty}}^2+ \frac{1}{\varepsilon^m}\,|v|^2_{{}_{H^1}}.
\end{equation*}
Again, by integrating in time and by summing on $k$, we end up with
\begin{equation*}
\begin{aligned}
|y-\partial_x z|_{{}_{L^2}} \leq&\, C \int_0^t \left \{ \Omega^\varepsilon(\xi) (1+ |v|^2_{{}_{L^2}}) +\left(1+\frac{1}{\varepsilon^m}\right)\, |v|^2_{{}_{H^1}} + \varepsilon^m \, |U^\varepsilon|^2_{{}_{L^\infty}} \right \} \, E_1(s,t) ds \\
&+C \int_0^t \left \{ \Omega^\varepsilon(\xi) (1+ |v|^2_{{}_{L^2}}) +\left(1+\frac{1}{\varepsilon^m}\right)\, |v|^2_{{}_{H^1}} + \varepsilon^m \, |U^\varepsilon|^2_{{}_{L^\infty}} \right \} \, \sum_{k \geq 2} E_k(s,t) ds.
\end{aligned}
\end{equation*}
Now, given $n>0$, let us set
\begin{equation*}
	N(t):= \frac{1}{\varepsilon^n} \sup_{s\in[0,t]} |v-z|_{{}_{H^1}}\,E_1(s,0),
\end{equation*} 
so that we have
\begin{equation}\label{est1}
\begin{aligned}
\frac{1}{\varepsilon^n}E_1(t,0)|v-z|_{{}_{L^2}} \leq& \ C \int_0^t \frac{\Omega^\varepsilon(\xi)}{\varepsilon^n} (t-s)^{-1/2} E_2(s,t) E_1(s,0) ds \\
&+ \int_0^t \frac{1}{\varepsilon^n}\left( |v-z|^2_{{}_{H^1}}+ |z|^2_{{}_{H^1}}\right) (t-s)^{-1/2} E_2(s,t) E_1(s,0) ds
\end{aligned}
\end{equation}
and
\begin{equation}\label{est2}
\begin{aligned}
&\frac{1}{\varepsilon^n}E_1(t,0)|y-\partial_x z|_{{}_{L^2}} \leq  
C \int_0^t \frac{\Omega^\varepsilon(\xi)}{\varepsilon^n}  \left \{E_1(s,0)+(t-s)^{-1/2} E_s(s,t)\, E_1(s,0)\right \} \,  ds\\
& +C \int_0^t \left \{ \left(\frac{1}{\e^n}+\frac{1}{\e^{n+m} } \right)\,\left( |v-z|^2_{{}_{H^1}} + |z|^2_{{}_{H^1}}\right)+\frac{1}{\varepsilon^{n-m}} \, |U^\varepsilon|^2_{{}_{L^\infty}} \right \} \, E_1(s,0) ds \\
& +C \int_0^t\left \{ \left(\frac{1}{\e^n}+\frac{1}{\e^{n+m}} \right)\!\!\left( |v-z|^2_{{}_{H^1}} \!+\! |z|^2_{{}_{H^1}}\right)\!+\!\frac{1}{\e^{n-m}} \, |U^\varepsilon|^2_{{}_{L^\infty}} \right \} (t-s)^{-1/2} E_s(s,t) E_1(s,0) ds. \\
\end{aligned}
\end{equation}
By summing \eqref{est1} and \eqref{est2} and since there hold
\begin{align*}
	&\int_0^t e^{(2\Lambda^\varepsilon_2-\Lambda_1^\varepsilon)s}\,ds
		\leq \int_0^t e^{\Lambda^\varepsilon_2 s}\,ds
		=\frac{1}{\Lambda_2^\varepsilon}(e^{\Lambda^\varepsilon_2 s}-1)\leq  \frac{1}{|\Lambda_2^\varepsilon|}, 
		\\
	&\int_0^t (t-s)^{-1/2}\,E_2(s,t)\,ds
		\leq \int_0^t (t-s)^{-1/2}\,e^{\Lambda_2^\varepsilon\,(t-s)}\,ds
		\leq \frac{1}{|\Lambda_2^\varepsilon|^{1/2}}, 
\end{align*}
we end up with the estimate $N(t) \leq A N^2(t)+B$, with
 \begin{equation*}
		\left\{\begin{aligned}
		A&:= \varepsilon^{n-m} \, E_1(0,t) (t + |\Lambda_2^\varepsilon|^{-1/2}) ,\\
		B&:=C|\Omega^\varepsilon|_{{}_{L^\infty}}E_1(t,0) \left( t + |\Lambda_2^\varepsilon|^{-1/2} \right)+ \varepsilon^{-n-m} |\Lambda_2^\varepsilon|^{-1} |v_0|^2_{{}_{H^1}} \\
		& \quad+ \varepsilon^{m-n} |U^\varepsilon|^2_{{}_{L^\infty}} E_1(t,0)\, (t +|\Lambda_2^\varepsilon|^{-1/2}).
		\end{aligned}\right.
\end{equation*}
We now use the precise distribution with respect to $\varepsilon$ of the eigenvalues of $\mathcal L^\varepsilon_\xi$. Precisely, in \cite{MS} it has been proved that $\lambda_k^\varepsilon \leq -c/\varepsilon$ for all $k \geq 2$ (see \cite[Corollary 4.5]{MS}). 

As a consequence $|\Lambda_2^\varepsilon| \sim \e^{-1}$; hence, if we require $m < 1$, for all $n>0$, there holds $N(t)< 2B$ that is
\begin{equation*}
|v-z|_{{}_{H^1}} \leq C |\Omega^\varepsilon|_{{}_{L^\infty}} + \Bigl( \varepsilon^{1-m} |v_0|^2_{{}_{H^1}}E_1(0,t)+\e^m\Bigr).
\end{equation*}
Precisely,  we can choose $m=1-\delta$, for some $\delta \in (0,1)$,  and the proof is completed.

\end{proof}

Again, \eqref{boundresto} can be used again to decoupled the nonlinear system \eqref{CS}; this leads to the following consequence of Theorem \ref{NT}.

\begin{corollary}\label{cor:metaL2}
Let hypotheses of Theorem \ref{NT} be satisfied and let us also assume
\begin{equation*}
	(\xi-\bar \xi)\,\theta^\varepsilon(\xi)<0\quad\textrm{ for any } \xi\in J,\,\xi\neq \bar \xi
	\qquad\textrm{ and }\quad
	{\theta^\varepsilon}'(\bar \xi)<0.
\end{equation*}
Then, for $\varepsilon$ and $|{v_0}|_{{}_{L^2}}$ sufficiently small, the solution $\xi(t)$ to \eqref{CS} satisfies
\begin{equation}\label{metaxi}
|\xi(t)- \bar \xi| \leq |\xi_0| e^{-\beta^\varepsilon t}, \qquad\lim_{\varepsilon \to 0} \beta^\varepsilon =0.
\end{equation}
\end{corollary}

\begin{proof}
For any initial datum $\xi_0$, the variable $\xi(t)$ is such that
\begin{equation*}
	\frac{d\xi}{dt}=\theta^\varepsilon(\xi)\bigl(1 + r\bigr)+ \rho^\varepsilon(\xi,{v}),
		\qquad\textrm{with}\quad
	|r|\leq |v_0|_{{}_{H^1}}^2 e^{-|\Lambda_2^\varepsilon| t}+ \bigl\{C\,|\Omega^\varepsilon|_{{}_{\infty}}+\varepsilon^{\delta} \,|{v}_0|^2_{{}_{H^1}}\bigr\}
\end{equation*}
and 
$$ |\rho^\varepsilon(\xi,{v})| \leq C |{v}|^2_{{}_{H^1}} \leq |{z}|^2_{{}_{H^1}}+ |R|^2_{{}_{H^1}} \leq |{v}_0|^2_{{}_{H^1}} e^{-|\Lambda_2^\varepsilon|t} + \varepsilon^{\delta} \, |{v}_0|^2_{{}_{H^1}}+C \, |\Omega^\varepsilon|_{{}_{\infty}}. $$
Hence, in the regime of small $\varepsilon$ the solution $\xi(t)$ has similar decay properties to those of the solution to the following reduced equation
\begin{equation*}
\frac{d\eta}{dt}=\theta^\varepsilon(\eta), \qquad  \eta(0)=\xi(0).
\end{equation*}
As a consequence, $\xi$ converges exponentially to $\bar \xi$ as $t\to+\infty$. More precisely there exists $\beta^\varepsilon = -{\theta^\varepsilon}'(\bar \xi)$, $\beta^\varepsilon \to 0$ as $\varepsilon \to 0$, such that 
\begin{equation*}
|\xi(t)- \bar \xi| \leq |\xi_0|e^{-\beta^\varepsilon t},
\end{equation*}
showing the exponentially slow motion of the shock layer position.

\end{proof}

Estimate \eqref{metaxi} shows the slow evolution of the interface location: precisely, the convergence of the shock towards its equilibrium position $\bar \xi$ is much slower as $\varepsilon$ becomes smaller.


\newpage


\begin{thebibliography}{99}



\bibitem{AlikBatFus91}
Alikatos N.D., Bates P.W., Fusco G.,
{\it Slow motion for the Cahn-Hilliard equation in one space dimension},
J. DIfferential Equations 90 (1991) no. 1, 81--135
%
\bibitem{BerKamSiv01}
Berestycki, H., Kamin S., Sivashinsky G.,
{\it Metastability in a  flame front evolution equation},
 Interfaces Free Bound. 3 (2001), no. 4, 361--392. 
%


\bibitem{CarrPego89}
Carr, J., Pego, R. L., 
{\it Metastable patterns in solutions of $u_t=\varepsilon^2 u_{xx}+f(u)$}, 
Comm. Pure Appl. Math. 42 (1989) no. 5, 523--576. 



\bibitem{FuscHale89}
Fusco, G., Hale, J. K.,
{\it Slow-motion manifolds, dormant instability, and singular perturbations},
J. Dynam. Differential Equations 1 (1989) no. 1, 75--94. 



\bibitem{JinXin95}
Jin S., Xin Z.,
{\it The relaxation schemes for systems of conservation laws in arbitrary space dimensions},
Comm. Pure Appl. Math. 48 (1995), no. 3, 235--276.

\bibitem{KreiKrei86}
Kreiss G., Kreiss H.-O.,
{\it Convergence to steady state of solutions of Burgers' equation},
Appl. Numer. Math. 2 (1986) no. 3-5, 161--179. 


\bibitem{LafoOMal94}
Laforgue J.G.L., O'Malley R.E. Jr., 
{\it On the motion of viscous shocks and the supersensitivity of their steady-state limits},
Methods Appl. Anal. 1 (1994), no. 4, 465--487. 

\bibitem{LafoOMal95}
Laforgue J.G.L., O'Malley R.E. Jr., 
{\it Shock layer movement for Burgers equation},
Perturbations  methods in physical mathematics (Troy, NY, 1993). SIAM J. Appl. Math. 55 (1995)  no. 2, 332--347.

\bibitem{MS}
Mascia C., Strani M.,
{\it Metastability for nonlinear parabolic equations with application to scalar conservation laws }, SIAM J. Math. Anal. 45 (2013), no 5, 3084--3113.

\bibitem{OttRez06}
Otto F., Reznikoff M.G.,
{\it Slow Motion of Gradient Flows},
J, Diff., Equations 237 (2006), 372--420.

\bibitem{Pazy83}
Pazy A. (1983).
{\it Semigroups of Linear Operators and Applications to Partial Differential Equations},  
Applied Mathematical Science 44, Springer-Verlag, New York.


\bibitem{Pego89}
Pego R.L.,
{\it Front migration in the nonlinear Cahn-Hilliard equation},
Proc. Roy. Soc. London Ser. A 422 (1989) no. 1863, 261--278.

\bibitem{ReynWard95}
Reyna L.G., Ward M.J.,
{\it On the exponentially slow motion of a viscous shock},
Comm. Pure Appl. Math. 48 (1995), no. 2, 79--120. 


\bibitem{Str12}
Strani M.,
{\it Slow motion of internal shock layers for the Jin-Xin system in one space dimension},
J. Dyn. Diff. Equations, to appear 2014. http://arxiv.org/abs/1207.2024

\bibitem{Str13}
Strani M.,
{\it Metastable dynamics of internal interfaces for a convection-reaction-diffusion equation},
submitted.

\bibitem{Str14}
Strani M.,
{\it Nonlinear metastability for a parabolic system of reaction-diffusion equations},
submitted.

\bibitem{SunWard99}
Sun X., Ward M. J., Russell R.,
{\it Metastability for a generalized Burgers equation with application to propagating flame fronts},
European J. Appl. Math. 10 (1999), no. 1, 27--53


\end{thebibliography}
\end{document}